\documentclass[a4paper,11pt]{article}
\usepackage{graphicx}
\usepackage{amsmath,amsthm,amssymb,enumerate,multirow,stmaryrd}
\usepackage{euscript,mathrsfs}
\usepackage{dsfont}
\usepackage{xcolor}
\usepackage{empheq}
\usepackage{geometry}

\usepackage{bm}
\usepackage{bbm}
\usepackage[linkcolor=blue,colorlinks=true]{hyperref}
\catcode`\@=11 \@addtoreset{equation}{section}

\catcode`\@=12

\usepackage{stmaryrd}
\allowdisplaybreaks

\usepackage[labelformat=simple]{subcaption}

\usepackage{enumitem}
\usepackage{url}

\usepackage{tikz}
\usetikzlibrary{patterns}
\usepackage[normalem]{ulem}
\usepackage{cancel}

\newtheorem{Theorem}{Theorem}[section]

\newtheorem{Lemma}[Theorem]{Lemma}
\newtheorem{Corollary}[Theorem]{Corollary}

\newtheorem{Definition}[Theorem]{Definition}

\newtheorem{Remark}[Theorem]{Remark}
\newcommand{\bRemark}[1]{
			\begin{Remark} \label{R#1} }
\newcommand{\eR}{\end{Remark}}

\newcommand{\ith}{i^{\rm th}}

\newcommand{\ve}{\bm{e}}

\newcommand{\tor}{\mathbb{T}^d}

\newcommand{\Piq}{\Pi_Q}

\newcommand{\ds}{\,{\rm d}S_x}

\newcommand{\grid}{{\cal T}_h}

\newcommand{\TS}{\Delta t}

\newcommand{\Divh}{{\rm div}_h}
\newcommand{\Gradh}{\nabla_h}

\newcommand{\Gradedge}{\nabla_\faces}

\newcommand{\co}[2]{{\rm co}\{ #1 , #2 \}}
\newcommand{\Ov}[1]{\overline{ #1 } }
\newcommand{\avs}[1]{\left\{\hspace{-3pt}\left\{ #1 \right\} \hspace{-3pt}\right\} }

\newcommand{\aleq}{\stackrel{<}{\sim}}

\newcommand{\vn}{\bm{n}}

\newcommand{\vc}[1]{{\bm #1}}

\newcommand{\Div}{{\rm div}_x}
\newcommand{\Grad}{\nabla_x}

\newcommand{\dx}{\,{\rm d} {x}}

\newcommand{\dt}{\,{\rm d} t }
\newcommand{\dxdt}{\,{\rm d} {x}{\rm d} t }

\newcommand{\jump}[1]{\left\llbracket#1\right\rrbracket}
\newcommand{\abs}[1]{| #1|}
\newcommand{\Abs}[1]{\left| #1 \right|}

\newcommand{\norm}[1]{\left\lVert#1\right\rVert}

\newcommand{\intTd}[1]{\int_{\tor} #1 \dx}

\newcommand{\intfacesint}[1]{\int_{\facesint}{ #1 \ds} }
\newcommand{\intTaufacesint}[1]{\int_0^{\tau}\int_{\facesint}{ #1 \ds} {\rm d} t }
\newcommand{\intfacesintB}[1]{\int_{\facesint}{\Big( #1 \Big) \ds}}

\newcommand{\intK}[1] {\int_{K} #1 \dx }

\newcommand{\intTauTd}[1]{ \int_0^\tau \int_{\tor} #1 \dxdt}
\newcommand{\intTauTdB}[1]{  \int_0^\tau \int_{\tor} \left( #1 \right) \dxdt}

\newcommand{\R}{\mathbb{R}}
\newcommand{\I}{\mathbb{I}}

\newcommand{\intTor}[1]{\int_{\tor} #1 \ \dx}

\def\softd{{\leavevmode\setbox1=\hbox{d}%
          \hbox to 1.05\wd1{d\kern-0.4ex{\char039}\hss}}}

\newcommand{\bdot}{\boldsymbol{\cdot}}

\definecolor{Cgrey}{rgb}{0.85,0.85,0.85}
\definecolor{Cblue}{rgb}{0.50,0.85,0.85}
\definecolor{Cred}{rgb}{1,0,0}
\definecolor{fancy}{rgb}{0.10,0.85,0.10}
\definecolor{forestgreen}{rgb}{0.13, 0.55, 0.13}
\definecolor{pinegreen}{rgb}{0.0, 0.47, 0.44}
\newcommand{\cred}{\color{red}}
\newcommand{\cblue}{\color{blue}}

\date{}
\newcommand{\pd}{\partial}

\newcommand{\faces}{\mathcal{E}}
\newcommand{\facesi}{\faces _i}
\newcommand{\facesK}{\faces(K)}
\newcommand{\facesKi}{\faces _i(K)}
\newcommand{\facesKj}{\faces _j(K)}
\newcommand{\facesint}{\faces}

\newcommand{\br}{ \nonumber \\ }
\newcommand{\character}{\mathds{1}}

\allowdisplaybreaks

\newcommand{\RevOne}[1]{{#1}}
\newcommand{\RevTwo}[1]{{#1}}

\usepackage{CJKutf8} \newcommand{\chinese}[1]{\begin{CJK*}{UTF8}{gbsn}\! #1 \end{CJK*}}

\begin{document}

\title{Convergence of a generalized Riemann problem scheme \\ for the Burgers equation}
\author{
M\'aria Luk\'a\v{c}ov\'a-Medvi\softd ov\'a$^{*}$
\and Yuhuan Yuan$^{\dagger,*}$
}


\maketitle

\centerline{$^*$ Institute of Mathematics, Johannes Gutenberg-University Mainz}
\centerline{Staudingerweg 9, 55 128 Mainz, Germany}
\centerline{lukacova@uni-mainz.de}

\medskip
\centerline{$^\dagger$School of Mathematics, Nanjing University of Aeronautics and Astronautics}
\centerline{Jiangjun Avenue No. 29, 211106 Nanjing, P. R. China}
\centerline{yuhuanyuan@nuaa.edu.cn}

%
%
%
%
%

\medskip

\begin{abstract}
In this paper we study the convergence of a second order finite volume approximation of the scalar conservation law. This scheme is based on the generalized Riemann problem (GRP) solver.
We firstly investigate the stability of the GRP scheme and find that it might be entropy unstable when the shock wave is generated.
By adding an artificial viscosity we propose a new stabilized GRP scheme. Under the assumption that numerical solutions are uniformly bounded, we prove consistency and  convergence of this new GRP method.
\end{abstract}



\section{Introduction}
We consider a scalar conservation law on a bounded periodic domain $\tor$
\begin{align}\label{PDE}
& \pd_t u +    \Div \vc{f}( u) = 0, \quad (t,x) \in (0,T) \times \tor, \quad d = 2,3 
\end{align}
subject to initial data
\begin{equation}\label{In}
u(0,x) = u_0 \quad \mbox{satisfying} \quad u_0 \in L^\infty(\tor;\R).
\end{equation}
Since discontinuities (shock waves) may develop after finite time even with smooth initial data,
we further require that the problem \eqref{PDE} is considered in a distributional sense and admits a mathematical entropy pair \RevTwo{$(\eta, \vc{q})$} such that $\eta(u)$ is a convex function of $u$ and the following entropy inequality holds in the distributional sense
\begin{align}\label{En}
& \pd_{t} \eta(u) + \Div \vc{q}(u) \leq 0.
\end{align}
\RevTwo{
The  entropy inequality \eqref{En} is a selection criterion that rules out nonphysical solutions. It is well-know that
for scalar multidimensional conservation laws unique global-in-time weak entropy solution exists~\cite{Kruzkov:1970}.  Similarly,  one-dimensional systems of hyperbolic conservation laws \cite{Bressan-Crasta-Piccoli:2000,Bressan-Lewicka:2000} are well-posed in the class of weak entropy solutions.} 
%
%
However, for multidimensional systems of hyperbolic conservation laws, there might exist infinitely many weak entropy solutions, see  \cite{DeLellis-Szekelyhidi:2010} for the incompressible Euler system and \cite{Chiodaroli-DeLellis-Kreml:2015,Feireisl-Klingenberg-Kreml-Markfelder:2020,Chiodaroli-Kreml-Macha-Schwarzacher:2021} for the compressible barotropic and full Euler systems.

In this paper we consider a second order finite volume scheme based on the generalized Riemann problem solver and show its 
convergence of  to a weak entropy solution of \eqref{PDE}-\eqref{En}.
Since the GRP scheme deeply relies on the problem considered,  hereafter we perform the analysis for the Burgers equation, i.e. \eqref{PDE}-\eqref{En} with
\begin{align}\label{Burgers-1}
&  \vc{f}(u) = f(u) \,  \I_{d\times 1} ,\quad   f(u) = \frac{u^2}{2}, \\
&  \eta(u) =  \frac{u^2}{2},  \quad \vc{q}(u) = q(u) \, \Bbb{I}_{d\times 1}, \quad  q(u) = \frac{u^3}{3}, \label{Burgers-2}
\end{align}
\RevTwo{where $\Bbb{I}_{d\times 1}$ is a $d\times 1$ matrix of ones.}
We point out that our analysis can be extended to general scalar hyperbolic conservation laws.
\RevTwo{
In contrast to the proof of the uniqueness due to Kru\v{z}kov \cite{Kruzkov:1970} our convergence analysis requires only to work with one entropy function \eqref{Burgers-2}.}
\RevOne{
We also note that our analysis is not restricted to a particular choice of limiter though a special sign condition is required \cite{Fjordholm3}. Moreover, the convergence results also hold for a piecewise constant version.
}

In the literature we can find several results on the convergence analysis of numerical schemes for scalar hyperbolic conservation laws, see, e.g.,
\cite{Kroner-Rokyta:1994, Kroner-Noelle-Rokyta:1995}, where the concept of measure-valued solutions and DiPerna's  result \cite{DiPerna:1985a} on the characterisation of the Young measure as a Dirac measure acting on a unique weak entropy solution is applied. Our convergence result presented in this paper can be seen as an extension of the above mentioned results to a truly multidimensional second order finite volume method based on the GRP solver. More precisely, we derive a new stabilized GRP scheme for the Burgers equation \eqref{PDE}-\eqref{Burgers-2} that satisfies the discrete entropy inequality and then prove its convergence to a unique weak entropy solution. In this context we refer to  \cite{BenArtzi-Falcovitz-Li:2009, BenArtzi-Li:2021} where the \RevOne{ Lax-Wendroff-type} theorem was proved for the second order finite volume method based on the GRP solver, however under an additional assumption controlling the growth of the total variation in time. In contrary to the latter results, we only use the  discrete entropy inequality that implies suitable weak BV estimates allowing us to show the  convergence of the GRP scheme.

We close the introductory section by mentioning the recent results on convergence of numerical methods for  multidimensional systems of hyperbolic conservation laws. Due to their ill-posedness in the class of weak entropy solutions, the concept of  generalized solutions, such as  measure-valued or dissipative (measure-valued) solutions have been used in the convergence analysis,
see  \cite{Fjordholm1, Fjordholm2, FLM1, FLM2, Lukacova-Yuan} as well as the monograph \cite{Feireisl-Lukacova-Mizerova-She:2021}. For the error analysis between a numerical solution and an exact strong solution of a multidimensional system of hyperbolic conservation laws we refer to \cite{Cances_Mathis_Seguin:2016, Jovanovic-Rohde:2006} and \cite{LSY}.

The rest of the paper is organized as follows.
Section~\ref{Num} introduces the notations and the GRP scheme.  In Section~\ref{Stability}~we show that this numerical scheme is not entropy stable, cf. Lemma~\ref{es-D}. By adding an artificial viscosity we propose a new stabilized GRP method and show that it satisfies the discrete entropy inequality.
Section~\ref{Convergence} is the heart of the paper presenting the main results: the consistency, the {\it weak} convergence to a weak solution and the {\it strong} convergence to a strong solution on its lifespan.

\section{Numerical method} \label{Num}
In this section we introduce a second order finite volume approximation originally proposed by Ben-Artzi, Li and Warnecke \cite{BenArtzi-Li-Warnecke:2006}, whose numerical flux is obtained by means of a local generalized Riemann problem solver.
See also \cite{BenArtzi-Li:2007,Li-Du:2016} for more details.
\subsection{Notations}
\paragraph{Mesh.}  Let $\grid$ be a uniform structured mesh with the mesh size $h$ formed by squares in 2D, and cubes in 3D such that $\tor = \bigcup_{K \in \grid} K, \, d = 2,3$.

We denote by $\faces$ the set of all faces of $\grid$, and by $\facesi$ the set of all faces that are orthogonal to the $\ith$ basis vector $\ve_i$ of the canonical system.   The set of all faces of an element $K$ is denoted by $\facesK$ and the set of the faces in the $\ith$-direction is defined by $\facesKi$.
We write $\sigma= K|L$ if $\sigma \in \faces$ is the common face of elements $K$ and $L$.
Further, we denote by $x_K$ and $|K|=h^d$ (resp. $x_\sigma$ and $|\sigma|=h^{d-1}$) the center and the Lebesgue measure of an element $K\in\grid$ (resp. a face $\sigma \in \faces$), respectively.
And $\vn$ is the outer normal vector to $\sigma$, see Figure \ref{figmesh}.

In what follows we use the abbreviated notations for integrals
\begin{align} \label{Int}
\intTor{\bdot} := \sum_{K\in\grid}\intK{\bdot} \quad \mbox{and} \quad \intfacesint{\bdot}:=\sum_{\sigma\in\facesint} \, \int_{\sigma}\bdot \dx .
\end{align}
In addition, we introduce the notation $A \aleq B$ if there exists a generic constant $C > 0$ independent on the mesh size $h$, such that $A \leq C\cdot B$ for $A,B\in \mathbb{R}$.
And we use the notation $A\approx B$  for $A \aleq B, B \aleq A$.

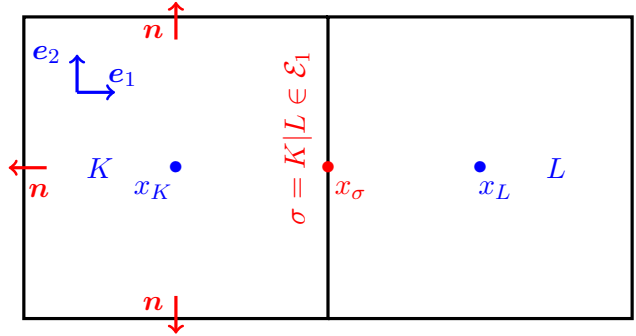
\begin{figure}[hbt]
\centering
\begin{tikzpicture}[scale=1.0]
\draw[-,very thick](0,-2)--(4,-2)--(4,2)--(0,2)--(0,-2)--(-4,-2)--(-4,2)--(0,2);
\draw[->,very thick, red = 90!] (-3.7,0)--(-4.2,0);
\path node at (-3.8,-0.3) {$\cred \vn$};
\draw[->,very thick, red = 90!] (-2,-1.7)--(-2,-2.2);
\path node at (-2.3,-1.8) {$\cred \vn$};
\draw[->,very thick, red = 90!] (-2,1.7)--(-2,2.2);
\path node at (-2.3, 1.8) {$\cred \vn$};

\draw[->,very thick, blue = 90!] (-3.3,1)--(-2.8,1);
\path node at (-2.7,1.2) {$\cblue \ve_1$};
\draw[->,very thick, blue = 90!] (-3.3,1)--(-3.3,1.5);
\path node at (-3.7,1.5) {$\cblue \ve_2$};

\path node at (-3.0,0) { $\cblue K$};
\path node at (3.0,0)  { $\cblue L$};
\path node at (-2.0,0) {$\cblue \bullet$};
\path node at (-2.3,-0.3) {$\cblue x_K$};
\path node at (2.0,0) {$\cblue \bullet$};
\path node at (2.2,-0.3) {$\cblue x_L$};
\path node at (0,0) {$\cred \bullet$};
\path node at (0.3,-0.3) {$\cred x_\sigma$};

\path (-0.4,0.4) node[rotate=90] { $\cred \sigma=K|L \in \faces_1$};
\end{tikzpicture}
\caption{ Structured mesh with $d = 2$.}\label{figmesh}
\end{figure}

 \paragraph{Function space.}
Let $P_n(K)$ \RevTwo{denote the set} of polynomials up to $n$-th degree on an element $K$, $n \in \mathbb{N}$.
We introduce the function space consisting of piecewise constant functions
\begin{equation}\label{Qh}
Q_h = \{ \phi \in L^1(\tor) \, \big| \,   \phi|_K \in P_0(K) \ \mbox{ for all } \ K \in \mathcal{T}_h \}
\end{equation}
and piecewise linear functions
\begin{equation}
V_h =\{ \phi \in L^2(\tor) \, \big| \,   \phi|_K \in P_1(K) \ \mbox{ for all } \ K \in \mathcal{T}_h \}.
\end{equation}
The notation $Q_h^m$ and $V_h^m$ stand for the corresponding $m$-dimensional vector function spaces, $m \in \mathbb{N}$.
The standard projection operator associated to $Q_h$ is defined as
\begin{equation}\label{Piq}
\Piq: L^1(\tor) \to Q_h, \quad \Piq  \phi (x) = \sum_{K \in \grid}  \frac{\character_{K}(x)}{|K|} \int_K \phi \dx,
\end{equation}
where $\character_{K} $ is the characteristic function given as
\begin{equation*}
\character_{K} (x)=
\begin{cases}
1 \ & \mbox{if} \  x \in K ,\\
0 & \mbox{otherwise}.
\end{cases}
\end{equation*}

In addition, for \RevTwo{$\sigma \in \facesKi$} we introduce the average and jump operators for $v \in V_h$
\[
\avs{v}_{\sigma}:= \frac{v^{\rm in}_{\sigma} + v^{\rm out}_{\sigma}  }{2},\quad
\jump{v}_{\sigma}:  = v^{\rm out}_{\sigma}  - v^{\rm in}_{\sigma}
\]
with
\begin{align}
& v^{\rm out}_{\sigma}: = \lim_{\delta \to 0+} v({ x_{\sigma}} + \delta \vc{n}_K),\quad
v^{\rm in}_{\sigma} : = \lim_{\delta \to 0+} v({ x_{\sigma}} - \delta \vc{n}_K),\\
& v_{\sigma}^+:= \lim_{\delta \to 0+} v({ x_{\sigma}} + \delta \RevTwo{\ve_i}),\quad \quad
v_{\sigma}^-: = \lim_{\delta \to 0+} v({ x_{\sigma}} - \delta \RevTwo{\ve_i}).
\end{align}
Obviously, it holds that
\begin{align}\label{In-Out}
v_\sigma^{\rm in} = \frac{2 \avs{v}_\sigma - \jump{v}_\sigma}{2}, \quad
v_\sigma^{\rm out} = \frac{2 \avs{v}_\sigma + \jump{v}_\sigma}{2}.
\end{align}

\paragraph{Discrete differential operators.}
The discrete differential operators are defined for the piecewise constant function $r_h \in Q_h$.
Firstly we define the discrete differential operator at interface $\sigma = K|L \in \faces$ by
\begin{align}
\left(\Gradedge r_h\right) _{\sigma} = \frac{\RevOne{\jump{r_h}_{\sigma}}}{h} \cdot \vn, \quad r_h \in Q_h. 
\end{align}
Having defined the {\em interfacial differential operator}  $\left(\Gradedge r_h\right) _{\sigma}$ we can proceed to define the discrete differential operator in an element $K$, denoted by $\left(\Gradh r_h\right) _{K}$, see Section \ref{sec:minmod} for further details.  Finally, we have the differential operator on $\tor$ that is defined elementwise
\begin{align}\label{DO}
&\Gradh r_h(x)  = \sum_{K\in\grid}  \left(\Gradh r_h\right) _{K}\character_{K}{(x)}, \quad
\Divh \vc{r}_h(x) = \mbox{tr}(\Gradh \vc{r}_h)(x),\quad \vc{r}_h
 \in Q_h^d.
 \end{align}
Here $\mbox{tr}(A)$ means the trace of the matrix $A$.

\paragraph{Reconstruction operator. }
Having the element differential operator $\Gradh$, we can define the following reconstruction operator for the piecewise constant function $r_h \in Q_h$
\begin{align}\label{RV}
& R_V: Q_h \to V_h, \quad R_V r_h (x) = \sum_{K \in \grid}  \character_{K}(x) \big( r_h + (x - x_K)\cdot \Gradh  r_h \big).
\end{align}


\paragraph{Time discretization.}
Given a time step $\TS>0$ we divide the time interval $[0,T]$ into $N_T=T/\TS$ uniform parts,  and denote $t_k= k\TS.$
Consider a pointwise function $v$, which is only known at time level $t_k,\, k=0,\dots,N_T$.
Let $L_{\TS}(0,T)$ \RevTwo{denote the space} of all piecewise constant in time functions $v$, such that
\begin{align*}
&  v(t) =v(t_0)  \ \mbox{ for } \  t < \Delta t, \quad  v(t)=v(t_k) \ \mbox{ for } \ t\in [k\TS,(k+1)\TS), \ \   k=1,\cdots, N_T.
\end{align*}
\RevOne{Combining with the definition of $Q_h$, we further denote the space of all piecewise constant in space and time functions by $L_{\TS}(0,T;Q_h)$.}
Further, we define the forward Euler time discretization operator $D_t$  as follows
\[
D_t v(t) = \frac{v (t + \TS) - v (t) }{\TS} .
\]
Hereafter, we shall write $v(t), D_t v(t)$ as $v, D_t v$ for simplicity if there is no confusion.

\subsection{Minmod limiter}\label{sec:minmod}
We now continue to explain the way how to define the element differential operator $\Gradh$ using  the interfacial  differential operator $\Gradedge$.
Here we use the Minmod limiter dimension by dimension, i.e. for any $r_h \in Q_h$
\RevTwo{
\begin{align}\label{Minmod}
 \left(\Gradh^{(i)} r_h\right) _{K} :=  \begin{cases}
\min_{\sigma \in \facesKi}  \left(\Gradedge^{(i)} r_h\right) _{\sigma} \ & \mbox{if }  \left(\Gradedge^{(i)} r_h\right) _{\sigma} > 0, \  \sigma \in \facesKi  , \\
\max_{\sigma \in \facesKi}  \left(\Gradedge^{(i)} r_h\right) _{\sigma} \ & \mbox{if }   \left(\Gradedge^{(i)} r_h\right) _{\sigma} < 0, \  \sigma \in \facesKi  , \\
0 & \mbox{otherwise}.
\end{cases}
\end{align}
}
Here the superscript $(i)$ means the $\ith$ component of a vector.
\RevOne{In what follows we derive a suitable property of the reconstructed solutions, which will be frequently used in the stability study in Section~\ref{Stability}.}

Let $\sigma := K|L$.
\RevOne{With \eqref{Minmod} we obtain}
\begin{subequations}\label{es-grad}
\begin{align}\label{es-grad-1}
& h \left(\Gradh^{(i)}  r_h\right) _{L} \in \co{0}{ \RevOne{\jump{r_h}_{\sigma}} }, \quad h \left(\Gradh^{(i)}  r_h\right) _{K} \in \co{0}{ \RevOne{ \jump{r_h}_{\sigma}} } , \\
& \co{A}{B} \equiv [ \min\{A,B\} , \max\{A,B\}].
\end{align}
\end{subequations}
\RevOne{In the sequel we shall omit the subscript $\sigma$ in $\jump{\cdot}_{\sigma}$ and $\avs{\cdot}_{\sigma}$ for simplicity, if there is no confusion.}
On the other hand,
with the definition of $R_V$, cf. \eqref{RV}, we have
\begin{subequations}\label{eq-min}
\begin{align}\label{eq-min-1}
& \jump{R_V r_h} = r_L - \frac{h}2  \left(\Gradh^{(i)}  r_h\right) _{L} - \left(r_K+\frac{h}2  \left(\Gradh^{(i)}  r_h\right) _{K} \right) = \jump{r_h} - h \avs{\Gradh^{(i)}  r_h},\\
& \avs{R_V r_h} = \avs{r_h} - \frac{h}4 \jump{\Gradh^{(i)}  r_h}.
\end{align}
\end{subequations}
Let us define two parameters
\begin{align}\label{factor}
& \lambda_{\sigma} := \begin{cases}
\frac{\jump{R_V r_h}}{\jump{r_h}}  & \mbox{ if } \jump{r_h} \neq 0, \\
1 & \mbox{ otherwise},
\end{cases}
\quad \quad
\mu_{\sigma} := \begin{cases}
\frac{\avs{r_h}}{\avs{R_V r_h}}  & \mbox{ if } \avs{R_V r_h}   \neq 0, \\
1,& \mbox{ if } \avs{r_h} = \avs{R_V r_h} = 0, \\
+\infty & \mbox{ otherwise}.
\end{cases}
\end{align}
Then \eqref{eq-min} together with \eqref{In-Out} give
\begin{subequations}\label{factor-1}
\begin{align}
& h\avs{\Gradh^{(i)}  r_h} = (1-\lambda_{\sigma}) \jump{r_h}, \quad h\jump{\Gradh^{(i)}  r_h} = 4(\mu_{\sigma}-1) \avs{R_V r_h}, \\
&    h\left(\Gradh^{(i)}  r_h\right) _{K} =  (1-\lambda_{\sigma}) \jump{r_h} - 2(\mu_{\sigma}-1) \avs{R_V r_h}, \\
&    h\left(\Gradh^{(i)}  r_h\right) _{L} =  (1-\lambda_{\sigma}) \jump{r_h} + 2(\mu_{\sigma}-1) \avs{R_V r_h}.
\end{align}
\end{subequations}
\RevTwo{Note that we shall understand $2(\mu_{\sigma}-1) \avs{R_V r_h}$ as $2\avs{r_h}$ given $\avs{R_V r_h} = 0,  \avs{ r_h} \neq 0$.}

Therefore, combining \eqref{es-grad} with \eqref{factor-1} we obtain the following property for the reconstructed solutions
\begin{align}\label{MinMod}
&\lambda_{\sigma} \in [0,1], \quad
(1-\lambda_{\sigma}) \jump{r_h} \pm 2(\mu_{\sigma}-1) \avs{R_V r_h} \in \co{0}{\jump{r_h}}.
\end{align}

\begin{Remark}\label{rmk-trivial}

\

\begin{itemize}
\item Given $\jump{r_h} = 0$, we have $\jump{r_h} = \jump{R_V r_h} = 0$ and $\avs{r_h} = \avs{R_V r_h}$ resulting to $\lambda_{\sigma} =\mu_{\sigma} = 1$.

\item Given $\jump{R_V r_h} = 0, \jump{r_h} \neq 0$, thanks to \eqref{es-grad-1} and \eqref{eq-min-1} we have $h \left(\Gradh^{(i)} r_h\right) _{L} = h \left(\Gradh^{(i)} r_h\right) _{\RevOne{K}} = \jump{r_h}$. Consequently, we have $\lambda_{\sigma} = 0, \, \mu_{\sigma} = 1$.
\end{itemize}

%
\end{Remark}

\RevTwo{
\begin{Remark}
With \eqref{RV} and \eqref{es-grad} it is easy to check
\begin{equation}\label{minmod-max}
(R_V u_h)_{\sigma}^{\pm} \in \co{u_K}{u_L}, \quad \avs{R_V u_h} \in \co{u_K}{u_L}, \quad  \avs{u_h} \in \co{u_K}{u_L}, \quad \sigma:=K|L.
\end{equation}
Consequently, we have the following estimates
\begin{align}\label{minmod-max-1}
\abs{(R_V u_h)_{\sigma}^{\pm}} \leq \max( \abs{u_K}, \abs{u_L} ), \quad \abs{\avs{R_V u_h}} \leq \max( \abs{u_K}, \abs{u_L} ), \quad \abs{\avs{u_h}} \leq \max( \abs{u_K}, \abs{u_L} ).
\end{align}
\end{Remark}
}

\subsection{Finite volume GRP method}
Letting $u_h \in  L_{\TS}(0,T;Q_h(\tor))$, the GRP scheme can be described as
\begin{subequations}\label{scheme-old-1}
\begin{align}\label{GRP}
& D_t u_h   +  \sum_{\sigma \in \facesK} \frac{|\sigma|}{|K|} f_{\sigma}^{GRP} \I_{d\times 1} \cdot \vn_K  = 0, \\
& f_{\sigma}^{GRP} := f(u^{\it RP}_{\sigma}) + \frac{\TS}{2}   \frac{\pd}{\pd u}f (u^{\it RP}_{\sigma}) \cdot (\pd_t u)^{\it GRP}_{\sigma}.  \label{GRP-flux}
\end{align}
\end{subequations}
Here, for $\sigma \in \facesKi$, $u^{\it RP}_{\sigma}$ is the solution at the face $\sigma$ of the local Riemann problem
\begin{equation*}
\begin{cases}
\pd_{t} u + \pd_{x^{(i)}} f(u) = 0, \quad t > \tau, \\
u(x^{(i)},\tau) = \begin{cases}
( R_V u_h )^-_{\sigma} & \mbox{ if }~ x^{(i)} < x_{\sigma}^{(i)}, \\
( R_V u_h )^+_{\sigma} & \mbox{ if }~ x^{(i)} > x_{\sigma}^{(i)}
\end{cases}
\end{cases}
\end{equation*}
and $(\partial_t u )^{\it GRP}_{\sigma} $ is obtained  by resolving analytically the local  quasi generalized Riemann problem
\begin{equation*}
\begin{cases}
\pd_t u + \pd_{x^{(i)}} f(u) = -\sum_{j\neq i}  \frac{\pd}{\pd u}f (u^{\it RP}_{\sigma}) \cdot (\pd_{x^{(j)}} u)_{\sigma}, \quad t > \tau, \\
(\pd_{x^{(j)}} u)_{\sigma}  = \begin{cases}
( \Gradh^{(j)} u_h )^-_{\sigma} & \mbox{if} \ u^{\it RP}_{\sigma} > 0,\\
( \Gradh^{(j)} u_h )^+_{\sigma} & \mbox{if} \ u^{\it RP}_{\sigma} < 0,
\end{cases}
\\
u(x,\tau) = R_Vu_h.
\end{cases}
\end{equation*}
As above the superscript $(i)$ means the $\ith$ component of a vector. For more details about the quasi generalized Riemann problem, we refer to the work of Li and Du \cite{Li-Du:2016}.
Table~\ref{t-GRP} lists all possible cases of $u^{\it RP}_{\sigma}$ and $(\partial_t u )^{\it GRP}_{\sigma}$.
\begin{table}[htpb]
	\centering
	\caption{Explicit expressions of $u^{\it RP}_{\sigma}$ and $(\partial_t u )^{\it GRP}_{\sigma}$.} \label{t-GRP}
	\begin{tabular}{|c|c|cc|}
		\hline
		 & Cases & $u^{\it RP}_{\sigma}$ &  $(\partial_t u )^{\it GRP}_{\sigma}$   \\
		\hline
		\hline
\multirow{2}{*}{Shock} &
$\jump{ R_V u_h } < 0,~ \avs{ R_V u_h } > 0$ &
$( R_V u_h )^{-}_{\sigma}$ &
$-( R_V u_h  )^{-}_{\sigma}\; (\Divh (u_h \I_{d\times 1}))_{\sigma}^-$  \\

&
$\jump{ R_V u_h } < 0,~ \avs{ R_V u_h } \leq 0$ &
$( R_V u_h )^{+}_{\sigma}$ &
$ -( R_V u_h  )^{+}_{\sigma}\; (\Divh (u_h \I_{d\times 1}))_{\sigma}^+$  \\

	\hline
	\hline
\multirow{3}{*}{Rarefaction} &
$ 0< ( R_V u_h )^{-}_{\sigma} \leq ( R_V u_h )^{+}_{\sigma}$ &
$( R_V u_h )^{-}_{\sigma}$ &
$ -( R_V u_h  )^{-}_{\sigma}\;  (\Divh (u_h \I_{d\times 1}))_{\sigma}^-$  \\

&
$( R_V u_h )^{-}_{\sigma} \leq 0 \leq ( R_V u_h )^{+}_{\sigma}$ &
$0$ &
$0$  \\

&
$( R_V u_h )^{-}_{\sigma} \leq ( R_V u_h )^{+}_{\sigma} < 0$ &
$( R_V u_h )^{+}_{\sigma} $ &
$ -( R_V u_h  )^{+}_{\sigma} \; (\Divh (u_h \I_{d\times 1}))_{\sigma}^+$   \\
		\hline
	\end{tabular}
\end{table}

Now let us rewrite the scheme \eqref{GRP} into the  weak form, i.e. for all $\phi_h \in Q_h$ we have
\begin{align} \label{scheme-old}
& \intTd{ \phi_h  D_t u_h}-  \intfacesint{ f_{\sigma}^{GRP}\jump{\phi_h}}= 0. 
\end{align}
According to Table~\ref{t-GRP} we obtain the expression of  the numerical flux $f_{\sigma}^{GRP} $
\begin{align}\label{fGRP}
 f_{\sigma}^{GRP} = \begin{cases}
\frac12 \abs{( R_V u_h )^{-}_{\sigma}}^2 \cdot \Big( 1 - \TS \cdot  (\Divh (u_h \I_{d\times 1}))_{\sigma}^- \Big) & \mbox{for Cases 1 and 3}, \\
\frac12 \abs{( R_V u_h )^{+}_{\sigma}}^2 \cdot \Big( 1 - \TS  \cdot (\Divh (u_h \I_{d\times 1}))_{\sigma}^+ \Big)  & \mbox{for Cases 2 and 5}, \\
0 & \mbox{for Case 4}.
 \end{cases}
\end{align}

\begin{Remark}\label{rmk-flux}
Denote
\begin{equation}\label{fRP}
f_{\sigma}^{RP}(R_V u_h) := f(u^{\it RP}_{\sigma})  \quad \mbox{ for all } \sigma \in \faces.
\end{equation}
On the one hand, we know from \eqref{scheme-old} that
it must be ensured that the term $f_{\sigma}^{RP}(R_V u_h)$ is entropy stable since $\Delta t$ could be as small as possible.

On the other hand, it is well-known that the Godunov flux is entropy stable. That is to say, we have
\begin{align}\label{DRP}
f_{\sigma}^{RP}(R_V u_h) - f_{\sigma}^{EC}(R_V u_h) = - D_{\sigma}^{RP}(R_V u_h) \jump{R_V u_h},
\end{align}
where $f_{\sigma}^{EC}$ is the second order entropy conservative numerical flux defined by
\begin{align}\label{EC-flux}
f_{\sigma}^{EC}(r_h) = \frac1{24} \bigg( 12 \abs{\avs{r_h}_\sigma}^2 + \abs{\jump{r_h}_\sigma}^2 \bigg), \quad r_h \in V_h,
\end{align}
and $D_{\sigma}^{RP} \geq 0$, see \eqref{DRP1} for the definition of $D_{\sigma}^{RP}$.

Hence, in what follows we study firstly the stability of $f_{\sigma}^{RP}(R_V u_h)$ with
\begin{align}\label{fRP-1}
f_{\sigma}^{RP}(R_V u_h)  - f_{\sigma}^{EC}(u_h)  &= f_{\sigma}^{RP}(R_V u_h) - f_{\sigma}^{EC}(R_V u_h)+ f_{\sigma}^{EC}(R_V u_h) - f_{\sigma}^{EC}(u_h) \br
&= - D_{\sigma}^{RP}(R_V u_h) \lambda_{\sigma} \jump{u_h} + f_{\sigma}^{EC}(R_V u_h)- f_{\sigma}^{EC}(u_h). 
\end{align}
Here $\lambda_{\sigma}$ is defined in \eqref{factor}.
\end{Remark}

\section{Stability}\label{Stability}
The goal of this section is to propose an entropy stable scheme based on the scheme \eqref{scheme-old}, and then prove the discrete entropy inequality.

\subsection{Properties of the ``RP" flux}\label{sec-RP}
Inspired by Remark \ref{rmk-flux}, we firstly analyze the property of the ``RP" flux $f_{\sigma}^{RP}(R_V u_h)$.
With \eqref{factor} and \eqref{EC-flux} we have
\begin{align*}
-f_{\sigma}^{RP}(R_V u_h)  + f_{\sigma}^{EC}(u_h) & = D_{\sigma}^{RP}(R_V u_h) \lambda \jump{u_h} + \frac{\mu^2 - 1}{2} \avs{R_V u_h}^2+ \frac{1-\lambda^2}{24} \jump{u_h}^2.
\end{align*}
Here and hereafter we shall write $\lambda_{\sigma}, \mu_{\sigma}$  as $\lambda, \mu$ for simplicity.
Applying Remark \ref{rmk-trivial}  we know that
\begin{itemize}
\item $\jump{u_h} = 0$: It holds $-f_{\sigma}^{RP}(R_V u_h)  + f_{\sigma}^{EC}(u_h) =0$.

\item $\jump{u_h} \neq 0, \jump{R_V u_h} = 0$: It holds $-f_{\sigma}^{RP}(R_V u_h)  + f_{\sigma}^{EC}(u_h) =   \frac{1}{24} \jump{u_h}^2.$
\end{itemize}
Hence, we define
\begin{equation}\label{D_RP}
D_{\sigma}^{(1)} := \begin{cases}
\frac{1}{24} \jump{u_h} &  \mbox{if} \ \RevTwo{\jump{R_V u_h} = 0},\\
\lambda D_{\sigma}^{RP}(R_V u_h)   + \frac{\mu^2 - 1}{2} \frac{\avs{R_V u_h}^2}{\jump{u_h}}+ \frac{1-\lambda^2}{24} \jump{u_h}&  \mbox{if} \  \jump{R_V u_h} \neq 0.
\end{cases}
\end{equation}
Obviously, $f_{\sigma}^{RP}(R_V u_h)$ is not entropy stable, since $D_{\sigma}^{(1)} < 0$ when $\jump{u_h} < 0, \jump{R_V u_h} = 0$. Thus, it is necessary to add a stabilized term, the so-called artificial viscosity, in $f_{\sigma}^{GRP}(R_V u_h)$.
To be more precise, in what follows we derive the estimates of $D_{\sigma}^{(1)}$.

\begin{Lemma}\label{es-D1}
For a numerical solution $u_h$ obtained by a finite volume method with an RP flux $f_{\sigma}^{RP}(R_V u_h)$ we have
\begin{align*}
\frac1{24} \jump{u_h} \leq D_{\sigma}^{(1)} \leq \frac{1}{12} \abs{\jump{u_h}} + \frac78 \RevOne{\max(\abs{u_K},\abs{u_L})}, \RevOne{\quad  \sigma:=K|L.}
\end{align*}
\end{Lemma}
\begin{proof}
Thanks to \eqref{D_RP}, in what follows we analyze $D_{\sigma}^{(1)}$ for $\jump{R_V u_h} \neq 0$ case by case.
Firstly, let us consider the cases generating a shock wave, i.e. $\jump{R_V u_h} = \lambda \jump{u_h}< 0$, cf. Cases 1 and 2 in Table \ref{t-GRP}.

\paragraph{Case 1 -- $\jump{ R_V u_h } < 0,~ \avs{ R_V u_h } > 0$:} Applying \eqref{factor} and \eqref{factor-1} yields
\begin{align*}
&\lambda \in (0,1], \quad
 \jump{u_h}\leq (1-\lambda) \jump{u_h} \pm 2(\mu-1) \avs{R_V u_h} \leq 0, \quad \jump{u_h}  < 0, \quad\avs{R_V u_h}  > 0.
\end{align*}
Consequently, we have
\begin{align}\label{case1}
&\lambda \in (0,1], \quad \jump{u_h}< 0, \quad  \avs{R_V u_h}  > 0, \quad
   2\abs{\mu-1}  \avs{R_V u_h} \leq  -\min(\lambda,1-\lambda) \jump{u_h}.
\end{align}
On the other hand, \RevTwo{with \eqref{DRP1}} we rewrite $D_{\sigma}^{(1)}$ as follows
\begin{align} \label{D-case3}
D_{\sigma}^{(1)} = & \frac{6 \lambda \avs{R_V u_h} - \lambda^2 \jump{u_h}}{12}  + \frac{\mu^2 - 1}{2} \frac{\avs{R_V u_h}^2}{\jump{u_h}}+ \frac{1-\lambda^2}{24} \jump{u_h} \nonumber \\
= & \frac{\lambda \avs{R_V u_h}}{2}  + \frac{\mu^2 - 1}{2} \frac{\avs{R_V u_h}^2}{\jump{u_h}}+ \frac{1-3\lambda^2}{24} \jump{u_h}.
\end{align}

Hence, combining \eqref{case1} with \eqref{D-case3} we obtain the upper bound  of $ D_{\sigma}^{(1)}$:
\begin{align*}
 D_{\sigma}^{(1)}
&\leq \frac{\avs{R_V u_h}}{2} + \frac{1}{12} \abs{\jump{u_h}} \quad\quad  \mbox{if }  \mu \in (-\infty,-1]\cup [1,\infty), \br
D_{\sigma}^{(1)}
& \leq  \frac{\avs{R_V u_h}}{2}   + \frac{\mu^2 - 1}{2} \frac{ \min(\lambda,1-\lambda) }{-2(1-\mu)} \cdot \avs{R_V u_h}+ \frac{1}{12} \abs{\jump{u_h}} \br
&\leq \frac{\avs{R_V u_h}}{2}   + \frac{1}{4} \avs{R_V u_h}+ \frac{1}{12} \abs{\jump{u_h}} = \frac34\avs{R_V u_h} + \frac{1}{12} \abs{\jump{u_h}} \quad\quad  \mbox{if }  \mu \in (-1,1).
\end{align*}
Further, letting $\chi := \frac{\avs{R_V u_h}}{\jump{u_h}}$ and
\begin{align}
g(\chi) := \frac{D_{\sigma}^{(1)}}{ \jump{u_h}} = \frac{\mu^2 - 1}{2} \chi^2 + \frac{\lambda}{2} \chi  +  \frac{1-3\lambda^2}{24},
\end{align}
in what follows we derive the lower bound of $ D_{\sigma}^{(1)}$ by analyzing the maximum of $g(\chi)$ in the range of
\begin{equation}\label{case1-1}
\{ \chi \, | \, \chi < 0, \  2\abs{\mu-1}  \chi \geq - \min(\lambda,1-\lambda)  \}  \ = \
\begin{cases}
\{ \chi \, | \, \chi < 0  \}  \ & \mbox{if } \ \mu = 1, \\
\left\{ \chi \, \big| \,  \frac{- \min(\lambda,1-\lambda) }{2\abs{\mu-1}} \leq \chi < 0 \right\} \ & \mbox{if } \ \mu \neq 1.
\end{cases}
\end{equation}
Specifically,
\begin{itemize}
\item If $\mu = 1$ then we have $  g(\chi) = \frac{\lambda}{2} \chi + \frac{1-3 \lambda^2}{24} < 1/24$.

\item If $\mu \in (-1,1)$, thanks to $\mu^2 - 1 < 0, \ -\frac{\lambda}{2(\mu^2 - 1)} > 0$ we have $  g(\chi)< g(0) = \frac{1-3 \lambda^2}{24} < 1/24$.

\item If $\mu \in (1,\infty)$, then it holds $ -\frac{ \min(\lambda,1-\lambda)}{2(\mu-1)}\leq  \chi < 0$ and
\begin{align}
g(0) = \frac{1-3 \lambda^2}{24}< \frac1{24}, \quad g\left( -\frac{ \lambda}{2(\mu-1)} \right) = \frac1{24}.
\end{align}
Thanks to $\mu^2 - 1 > 0, \ -\frac{\lambda}{2(\mu^2 - 1)} < 0$ we get  $g(\chi) \leq \max \left( g\left( -\frac{ \lambda}{2(\mu-1)} \right), g(0) \right)= \frac1{24}$.

\item If $\mu \in (-\infty,-1)$, then it holds $  -\frac{ \min(\lambda,1-\lambda)}{2(1-\mu)}\leq  \chi < 0$ and
\begin{align}
g(0) = \frac{1-3 \lambda^2}{24}< \frac1{24}, \quad g\left( -\frac{ \lambda}{2(1-\mu)} \right) = \frac1{24} - \frac{\lambda^2}{2(1-\mu)} < \frac1{24}.
\end{align}
Thanks to  $\mu^2 - 1 > 0, \ -\frac{\lambda}{2(\mu^2 - 1)} < 0$ we have $g(\chi) \leq \max \left( g\left( -\frac{ \lambda}{2(1-\mu)} \right), g(0) \right)< \frac1{24}$.
\end{itemize}

Altogether, we have in the Case 1: 
\begin{align*}
\frac1{24} \jump{u_h} \leq D_{\sigma}^{(1)} \leq \frac34\avs{R_V u_h} + \frac{1}{12} \abs{\jump{u_h}}.
\end{align*}

\paragraph{Case 2 -- $ \jump{ R_V u_h } < 0,~ \avs{ R_V u_h } \leq 0$:} Analogously to Case 1 we obtain \RevTwo{$\frac1{24} \jump{u_h} \leq D_{\sigma}^{(1)} \leq \frac34\abs{\avs{R_V u_h}} + \frac{1}{12} \abs{\jump{u_h}}$}. We leave the details for interested readers.

\

Next, let us consider the cases generating a rarefaction wave: $\jump{R_V u_h} = \lambda \jump{u_h}> 0$, i.e. Cases 3, 4 and 5 in Table \ref{t-GRP}.

\paragraph{Case 3 -- $ 0< ( R_V u_h )^{-}_{\sigma} < ( R_V u_h )^{+}_{\sigma}$:}  Applying \eqref{factor}, \eqref{factor-1} and
\begin{align}\label{eq01}
& ( R_V u_h )^{-}_{\sigma}  = \avs{R_V u_h } - \frac{\lambda}2 \jump{u_h}, \quad
 ( R_V u_h )^{+}_{\sigma}  = \avs{R_V u_h } + \frac{\lambda}2 \jump{u_h},
\end{align}
we have
\begin{align*}
&\lambda \in (0,1], \quad
 0 \leq (1-\lambda) \jump{u_h} \pm 2(\mu-1) \avs{R_V u_h} \leq  \jump{u_h}, \quad
 2\avs{R_V u_h} >  \lambda \jump{u_h} > 0
\end{align*}
resulting to
\begin{align}\label{case3}
&\lambda \in (0,1], \quad \mu \in (0,2), \quad 2 \avs{R_V u_h} > \lambda \jump{u_h} > 0, \quad
   2\abs{\mu-1}  \avs{R_V u_h} \leq  \min(\lambda,1-\lambda) \jump{u_h}.
\end{align}
On the other hand, thanks to \eqref{D_RP} and \eqref{DRP1} we know that the expression of $D_{\sigma}^{(1)}$ is the same as the one in Case 1.

Hence, combining \eqref{D-case3} with \eqref{case3} we obtain the upper bound of $ D_{\sigma}^{(1)}$:
\begin{align*}
 D_{\sigma}^{(1)}
&\leq  \frac{\avs{R_V u_h}}{2}  + \frac{1}{24} \jump{u_h}  \quad\quad  \mbox{if }  \mu \in (0,1],
\br
D_{\sigma}^{(1)}
& \leq  \frac{\avs{R_V u_h}}{2}   + \frac{\mu^2 - 1}{2} \frac{ \min(\lambda,1-\lambda) }{2(\mu-1)} \cdot \avs{R_V u_h}+ \frac{1}{24} \jump{u_h} \br
&\leq \frac{\avs{R_V u_h}}{2}   + \frac{(\mu+1)}{4} \min(\lambda,1-\lambda)  \avs{R_V u_h}+ \frac{1}{24} \jump{u_h} \br
& \leq \frac{\avs{R_V u_h}}{2}   + \frac{3}{4} \cdot \frac12  \avs{R_V u_h}+ \frac{1}{24} \jump{u_h} = \frac78 \avs{R_V u_h}+ \frac{1}{24} \jump{u_h} \quad\quad  \mbox{if }  \mu \in (1,2).
\end{align*}
On the other hand, we study the lower bound of $ D_{\sigma}^{(1)}$ by analyzing the minimum of $g(\chi)$ by in the range of
\begin{equation}
\left\{ \chi \, | \, \chi > \frac{\lambda}{2}, \  2\abs{\mu-1}  \chi \leq \min(\lambda,1-\lambda)  \right\}  \ = \
\begin{cases}
\{ \chi \, | \, \chi >  \lambda/2 \}  \ & \mbox{if } \ \mu = 1, \\
\left\{ \chi \, \big| \,  \lambda/2 < \chi \leq \frac{ \min(\lambda,1-\lambda) }{2\abs{\mu-1}} \right\} \ & \mbox{if } \ \mu \neq 1.
\end{cases}
\end{equation}
Specifically,
\begin{itemize}
\item If $\mu \in [1,2)$ then we have $g(\chi) \geq \frac{\lambda}{2} \chi + \frac{1-3 \lambda^2}{24} > \frac{\lambda^2}{4} + \frac{1-3 \lambda^2}{24} \geq 1/24$.

\item If $\mu \in (0,1)$, then it holds
\begin{align*}
 & g\left(\frac{\lambda}2 \right) = \frac{1+\RevTwo{3\lambda^2 \mu^2}}{24} > \frac{1}{24}, \quad   g\left(\frac{\lambda}{2(1-\mu)} \right) = \frac{1}{24}.
\end{align*}
Thanks to $\frac{\mu^2 - 1}{2} < 0, \, \frac{\lambda}{2(1-\mu^2)} > 0$ we obtain $g(\chi) \geq \frac{1}{24}$.
\end{itemize}

Altogether, we  have in the Case 3: 
\begin{align*}
\frac1{24} \jump{u_h} \leq D_{\sigma}^{(1)} \leq \frac78 \avs{R_V u_h}+ \frac{1}{24} \jump{u_h}.
\end{align*}

\paragraph{Case 4 -- $( R_V u_h )^{-}_{\sigma} < 0 < ( R_V u_h )^{+}_{\sigma}$:}
\
Applying \eqref{factor}, \eqref{factor-1} and \eqref{eq01} we obtain
\begin{align*}
\jump{R_V u_h} \geq 2\abs{\avs{R_V u_h} }, \ \jump{u_h} \geq 2\abs{\avs{ u_h} }.
\end{align*}
On the other hand, \RevTwo{with \eqref{DRP1}} we have
\begin{align*}
D_{\sigma}^{(1)} = & \frac{12 \avs{R_V u_h}^2 + \lambda^2\jump{u_h}^2}{24\jump{u_h}}  + \frac{\mu^2 - 1}{2} \frac{\avs{R_V u_h}^2}{\jump{u_h}}+ \frac{1-\lambda^2}{24} \jump{u_h}= \frac{\mu^2}{2} \frac{\avs{R_V u_h}^2}{\jump{u_h}}+ \frac{1}{24} \jump{u_h}.
\end{align*}
Consequently, we have the following estimate of $D_{\sigma}^{(1)}$:
\begin{align*}
\frac{1}{24}  \RevTwo{\jump{u_h}} \leq  D_{\sigma}^{(1)} = \frac{\avs{ u_h}^2}{2\jump{u_h}}+ \frac{1}{24} \jump{u_h}\leq  \frac14 \abs{\avs{ u_h} }+ \frac{1}{24} \jump{u_h}. 
\end{align*}

\paragraph{Case 5 -- $ ( R_V u_h )^{-}_{\sigma} < ( R_V u_h )^{+}_{\sigma} < 0$:} Analogously to Case 3, we obtain \RevTwo{$\frac1{24} \jump{u_h} \leq D_{\sigma}^{(1)} \leq \frac78 \abs{\avs{R_V u_h}}+ \frac{1}{24} \jump{u_h}$}. 

\

In total, \RevTwo{together with \eqref{minmod-max-1} we obtain}
\begin{align*}
\frac1{24} \jump{u_h} \leq D_{\sigma}^{(1)} \leq \frac78 \RevOne{\max(\abs{u_K},\abs{u_L})} + \frac{1}{12} \abs{\jump{u_h}}
\end{align*}
and finish the proof.
\end{proof}

\subsection{Properties of the GRP flux}\label{sec-GRP}

Now we proceed to study the GRP flux \eqref{fGRP},  i.e. for $\sigma \in \facesKi$
\begin{align}
&  f_{\sigma}^{EC}(u_h) - f_{\sigma}^{GRP} =  f_{\sigma}^{EC}(u_h)  -f_{\sigma}^{RP}(R_V u_h) +  f_{\sigma}^{RP}(R_V u_h) - f_{\sigma}^{GRP}   \br
= \ &  D^{(1)}_{\sigma} \jump{u_h}_{\sigma} + \frac{\TS}{2h} \cdot
 \begin{cases}
\abs{( R_V u_h )^{-}_{\sigma}}^2 \cdot   (h\Divh (u_h \I_{d\times 1}))_{\sigma}^-  & \mbox{for Cases 1 and 3}, \\
\abs{( R_V u_h )^{+}_{\sigma}}^2 \cdot   (h\Divh (u_h \I_{d\times 1}))_{\sigma}^+   & \mbox{for Cases 2 and 5}, \\
0 & \mbox{for Case 4}.
 \end{cases}
\end{align}
Thanks to the definition of Minmod limiter \eqref{Minmod} we know that  $(h \Gradh^{(i)} u_h)_K$ has the same sign as $\jump{u_h}_{\sigma}$, and $(h \Gradh^{(j)} u_h)_K,j\neq i,$ has the same sign as $\jump{u_h}_{\hat{\sigma}}, \hat{\sigma} \in \facesKj, \sigma \cap\hat{\sigma} \neq \emptyset$.
Hence, we \RevTwo{rewrite} $ f_{\sigma}^{EC}(u_h) - f_{\sigma}^{GRP} $ as
\begin{align}\label{fGRP-1}
 f_{\sigma}^{EC}(u_h) - f_{\sigma}^{GRP} & = D^{(1)}_{\sigma} \jump{u_h}_{\sigma} + D^{(2)}_{\sigma} \jump{u_h}_{\sigma} + \sum_{j\neq i} D_{\sigma,\hat{\sigma}} \jump{u_h} _{\hat{\sigma}} \RevTwo{=} D_{\sigma}  \jump{u_h}_{\sigma} + \sum_{j\neq i} D_{\sigma,\hat{\sigma}} \jump{u_h} _{\hat{\sigma}},
\end{align}
\RevOne{
where $D_{\sigma}:= D^{(1)}_{\sigma} + D^{(2)}_{\sigma} $, $D^{(1)}_{\sigma}$ is defined in \eqref{D_RP} and
\begin{align*}
& D^{(2)}_{\sigma} \jump{u_h}_{\sigma} :=  \frac{\TS}{2h}
 \begin{cases}
\abs{( R_V u_h )^{-}_{\sigma}}^2 \cdot   (h \Gradh^{(i)} u_h)_{\sigma}^-  & \mbox{for Cases 1 and 3}, \\
\abs{( R_V u_h )^{+}_{\sigma}}^2 \cdot   (h \Gradh^{(i)} u_h)_{\sigma}^+   & \mbox{for Cases 2 and 5}, \\
0 & \mbox{for Case 4},
 \end{cases}
 \br
& D_{\sigma,\hat{\sigma}} \jump{u_h} _{\hat{\sigma}} :=  \frac{\TS}{2h}
 \begin{cases}
\abs{( R_V u_h )^{-}_{\sigma}}^2 \cdot   (h \Gradh^{(j)} u_h)_{\hat{\sigma}}^-  & \mbox{for Cases 1 and 3}, \\
\abs{( R_V u_h )^{+}_{\sigma}}^2 \cdot   (h \Gradh^{(j)} u_h)_{\hat{\sigma}}^+   & \mbox{for Cases 2 and 5}, \\
0 & \mbox{for Case 4}.
 \end{cases}
\end{align*}}

\RevTwo{
Combining with \eqref{es-grad}, \eqref{minmod-max-1} and Lemma \ref{es-D1} we finally obtain the estimates of $D_{\sigma} $ and $D_{\sigma,\hat{\sigma}}, \hat{\sigma} \neq \sigma$.}
\begin{Lemma}\label{es-D}
For a numerical solution $u_h$ obtained by the finite volume method with the GRP numerical flux \eqref{fGRP} we have
\begin{align*}
& \frac1{24} \jump{u_h} \leq  D_{\sigma} \leq  \frac{1}{12} \abs{\jump{u_h}} + \frac78 \RevOne{\max(\abs{u_K},\abs{u_L})}+ \frac{\TS}{2h} \RevOne{\max(u_K^2, u_L^2}),\quad
  \abs{D_{\sigma,\hat{\sigma}}} \leq \frac{\TS}{2h} \RevOne{\max(u_K^2,u_L^2)}
\end{align*}
\RevOne{with $\sigma:=K|L \in  \facesKi$, $\hat{\sigma} \in \facesKj, j \neq i$ and $\sigma \cap\hat{\sigma} \neq \emptyset$.}
\end{Lemma}

\begin{Remark}
Note that the upper bound of $D_{\sigma}$ can be formulated with the functions $\avs{u_h}$ and $\jump{u_h}$
\begin{align*}
D_{\sigma} \aleq \abs{\avs{u_h}} + \abs{\jump{u_h}}  + \frac{\TS}{h} \left( \abs{\avs{u_h}}^2 + \abs{\jump{u_h}}^2 \right),
\end{align*}
since $\RevOne{\max(\abs{u_K},\abs{u_L})}  \leq \abs{\avs{u_h}} + \frac12 \abs{\jump{u_h}}$.
Furthermore, we can observe that, the upper bound of $D_{\sigma} $ is almost the same as that of $D_{\sigma}^{RP}$, cf. \eqref{es-DRP}.

\

However, the lower bounds of  $D_{\sigma} $ and $D_{\sigma}^{RP}$ are different. This is because the GRP flux is a second order approximation of $f(u)$ while the RP flux  is first order.
\end{Remark}

\subsection{Stabilized scheme}\label{sec-Mod}
For $\jump{u_h}<0$ we need to add an artificial diffusion term to the GRP flux. Our new GRP scheme can be formulated as follows
\begin{subequations}\label{scheme}
\begin{align}\label{GRP-2}
& \intTd{ \phi_h  D_t u_h}-  \intfacesint{ f_{\sigma}^{GRP,New}\jump{\phi_h}}= 0,  \\ \label{flux-2}
& f_{\sigma}^{GRP,New} = \begin{cases}
f_{\sigma}^{GRP}  & \ \mbox{if} \ \jump{u_h} \geq 0, \\
f_{\sigma}^{GRP} + \left(\frac{1}{24}+C_1\right) \abs{\jump{u_h}}^2 & \ \mbox{if} \ \jump{u_h} < 0,
\end{cases}\quad u_h(0,x) = \Piq u_0,
\end{align}
\end{subequations}
\RevOne{where $C_1$ is a constant satisfying $C_1\in(0, 1/24]$. Note that $C_1$ is independent on $\TS$ and $h$.}
Denoting
\begin{align}\label{flux-new}
D_{\sigma}^{New} \jump{u_h}  :=   -f_{\sigma}^{GRP,New} + f_{\sigma}^{EC}(u_h) - \sum_{j\neq i} D_{\sigma,\hat{\sigma}} \jump{u_h} _{\hat{\sigma}}
\end{align}
from Lemma \ref{es-D} we obtain  the following estimate 
\begin{align}\label{es-D-1}
\RevOne{C_1 \cdot \abs{\jump{u_h}} \leq D_{\sigma}^{New} \leq \frac{1}{6} \abs{\jump{u_h}} + \frac78 \RevOne{\max(\abs{u_K},\abs{u_L})}+ \frac{\TS}{2h} \RevOne{\max(u_K^2, u_L^2}).}
\end{align}
Accordingly, the modification \eqref{flux-2} may lead to the entropy stability of the new GRP method. Hereafter we analyze this new stabilized GRP method \eqref{scheme}.

\subsection{Entropy stability}\label{sec-ES}
Before analyzing the discrete entropy inequality of the new stabilized GRP scheme \eqref{scheme}, we formulate a physically reasonable assumption
\begin{align}\label{assumption}
0 \leq \abs{u_h} \leq \Ov{u} \quad \quad \mbox{ uniformly for } \quad \TS,h \to 0.
\end{align}
Consequently, it holds
\begin{align}\label{assumption-1}
\abs{( R_V u_h )^{\pm}_{\sigma}} \leq \Ov{u}, \quad
\abs{D_{\sigma}^{New}} \leq  \frac{29}{24} \Ov{u}+ \frac{\TS}{2 h} \Ov{u}^2, \quad
 \abs{D_{\sigma,\hat{\sigma}} } \leq \frac{\TS}{2 h} \Ov{u}^2.
\end{align}

Further, let us start by deriving the entropy balance of the \RevTwo{stabilized} GRP scheme \eqref{scheme}.

\begin{Lemma}[Entropy balance]\label{lem-EB}
Let $u_h$ be the numerical solution of the \RevTwo{stabilized} GRP method \eqref{scheme}. Then it holds for all $\varphi_h \in Q_h$
\begin{align}\label{EB}
& \intTd{ \varphi_h  D_t \eta_h}- \intfacesint{ q_{\sigma}^{GRP,New} \jump{\varphi_h} } =  - \intfacesint{ D_{\sigma}^{New}  \jump{u_h}^2 \avs{\varphi_h}} \br
&- \RevTwo{ \sum_{i=1}^d\int_{\sigma\in\mathcal{E}_i} \sum_{j\neq i, \hat{\sigma}\in \mathcal{E}_j} } D_{\sigma,\hat{\sigma}} \jump{u_h}_{\sigma} \jump{u_h}_{\hat{\sigma}}  \avs{\varphi_h} \ds+ \frac{\TS}{2 |K|} \sum_{K} \varphi_K \left(\sum_{\sigma \in \facesK} |\sigma| \left( f_{\sigma}^{GRP} - f(u_K) \right)\I_{d\times 1} \cdot \vn_K  \right)^2
\end{align}
with
\begin{align*}
\eta_h = \eta(u_h),\quad \psi_h = \psi(u_h), \quad \psi = u^3/6,\quad q_{\sigma}^{GRP,New} = \avs{u_h} f_{\sigma}^{GRP,New} - \avs{\psi_h}.
\end{align*}
\RevTwo{Note that $\hat{\sigma}$ is a neighbour of $\sigma$ in the sense that $\sigma \cap \hat{\sigma} \neq \emptyset$.}
\end{Lemma}

\begin{proof}
Taking the test function in \eqref{GRP-2} as $\varphi_h u_h$ we have
\begin{align*}
\intTd{ \varphi_h  u_h D_t u_h}-  \intfacesint{ f_{\sigma}^{GRP,New}\jump{\varphi_h u_h}}=  0.
\end{align*}
On the one hand,  with \eqref{flux-new} and $\intfacesintB{ \avs{\psi_h} \jump{\varphi_h} +\jump{\psi_h} \avs{\varphi_h}  } = 0$, $\jump{\psi_h} =  f_{\sigma}^{EC}(u_h) \jump{u_h}$  we reformulate $\intfacesint{ \RevTwo{f_{\sigma}^{GRP,New}}\jump{\varphi_h u_h}}$ as
\begin{align*}
&\intfacesint{ \RevTwo{f_{\sigma}^{GRP,New}}\jump{\varphi_h u_h}}= \intfacesint{ f_{\sigma}^{GRP,New}(\jump{\varphi_h} \avs{u_h} + \avs{\varphi_h} \jump{u_h})} \\
=& \intfacesint{ q_{\sigma}^{GRP,New} \jump{\varphi_h} } + \intfacesint{ \avs{\psi_h} \jump{\varphi_h} } + \intfacesint{ f_{\sigma}^{EC}(u_h)  \jump{u_h} \avs{\varphi_h}} \\
&- \intfacesint{ D_{\sigma}^{New}  \jump{u_h}^2 \avs{\varphi_h}} - \RevTwo{ \sum_{i=1}^d\int_{\sigma\in\mathcal{E}_i} \sum_{j\neq i, \hat{\sigma}\in \mathcal{E}_j} } D_{\sigma,\hat{\sigma}} \jump{u_h}_{\sigma}  \jump{u_h}_{\hat{\sigma}} \avs{\varphi_h} \ds \\
=& \intfacesint{ q_{\sigma}^{GRP,New} \jump{\varphi_h} } + \intfacesint{ \left( -\jump{\psi_h} + f_{\sigma}^{EC}(u_h)  \jump{u_h} \right) \avs{\varphi_h}} \\
&- \intfacesint{ D_{\sigma}^{New}  \jump{u_h}^2 \avs{\varphi_h}} - \RevTwo{ \sum_{i=1}^d\int_{\sigma\in\mathcal{E}_i} \sum_{j\neq i, \hat{\sigma}\in \mathcal{E}_j} } D_{\sigma,\hat{\sigma}} \jump{u_h}_{\sigma} \jump{u_h}_{\hat{\sigma}}  \avs{\varphi_h} \ds \\
=& \intfacesint{ q_{\sigma}^{GRP,New} \jump{\varphi_h} } - \intfacesint{ D_{\sigma}^{New}  \jump{u_h}^2 \avs{\varphi_h}} - \RevTwo{ \sum_{i=1}^d\int_{\sigma\in\mathcal{E}_i} \sum_{j\neq i, \hat{\sigma}\in \mathcal{E}_j} } D_{\sigma,\hat{\sigma}} \jump{u_h}_{\sigma} \jump{u_h}_{\hat{\sigma}}   \avs{\varphi_h} \ds.
\end{align*}
On the other hand, we rewrite \eqref{GRP-2} as
\begin{align}
&D_t u_h =  - \sum_{\sigma \in \facesK} \frac{|\sigma|}{|K|} \left( f_{\sigma}^{GRP,New} - f(u_K) \right)\I_{d\times 1} \cdot \vn_K  
\end{align}
and then get
\begin{align*}
\intTd{\varphi_h(D_t \eta_h - u_h D_t u_h)} & =   \sum_{K} |K| \cdot \frac{\TS}{2 }\abs{D_t u_K}^2 \varphi_K \br
&= \frac{\TS}{2 } \sum_{K} |K| \varphi_K \left(\sum_{\sigma \in \facesK} \frac{|\sigma|}{|K|} \left( f_{\sigma}^{GRP} - f(u_K) \right)\I_{d\times 1} \cdot \vn_K  \right)^2.
\end{align*}
Combining both we finish the proof.
\end{proof}

\

Next, we derive the uniform bounds from the entropy balance \eqref{EB}.

\begin{Lemma}[Uniform bounds]\label{lem-UB}
Let $u_h$ be the numerical solution of the \RevTwo{stabilized} GRP method \eqref{scheme} with
\begin{align}\label{assumption-h-dt}
(\TS,h) \in (0,1]^2, \quad \TS \leq  C_2 \, h^{4/3},  \quad  C_2 \mbox{ is a constant independent of } \TS \mbox{ and } h.
\end{align}
Under assumption \eqref{assumption}  it holds
\begin{align}\label{UB1}
\intTaufacesint{  \abs{\jump{u_h}}^3 }  \aleq 1.
\end{align}
\end{Lemma}

\begin{proof}
Taking the test function in \eqref{EB} as $\varphi_h \equiv 1$ we have
\begin{align}\label{EB-1}
\intTd{ D_t \eta_h}&=  - \intfacesint{ D_{\sigma}^{New}  \jump{u_h}^2 } - \RevTwo{ \sum_{i=1}^d\int_{\sigma\in\mathcal{E}_i} \sum_{j\neq i, \hat{\sigma}\in \mathcal{E}_j} } D_{\sigma,\hat{\sigma}} \jump{u_h}_{\sigma} \jump{u_h}_{\hat{\sigma}} \ds \br
&+ \frac{\TS}{2 |K|} \sum_{K} \left(\sum_{\sigma \in \facesK} |\sigma| \left( f_{\sigma}^{GRP} - f(u_K) \right)\I_{d\times 1} \cdot \vn_K  \right)^2.
\end{align}
The first term on the right hand side of \eqref{EB-1} can be estimated as
\begin{align}
\intfacesint{ D_{\sigma}^{New}  \jump{u_h}^2 } \geq  C_1 \intfacesint{ \abs{\jump{u_h}}^3}.
\end{align}
Thanks to \eqref{assumption-1} we control the second term with
\begin{align*}
&\Abs{ \RevTwo{ \sum_{i=1}^d\int_{\sigma\in\mathcal{E}_i} \sum_{j\neq i, \hat{\sigma}\in \mathcal{E}_j} } D_{\sigma,\hat{\sigma}} \jump{u_h}_{\sigma} \jump{u_h}_{\hat{\sigma}} \ds }
\aleq \frac{\TS}{h}   \intfacesint{  \Abs{\jump{u_h}_{\sigma}}^2 }.
\end{align*}
\RevTwo{With \eqref{flux-new}, \eqref{assumption-1}} and
\begin{align*}
\abs{f_{\sigma}^{EC}(u_h)- f(u_K)} = \Abs{\frac{-\jump{u_h}+6 \avs{u_h}}{12} \jump{u_h}} \aleq \abs{\jump{u_h}},\quad
\abs{f_{\sigma}^{EC}(u_h)- f(u_L)} \aleq \abs{\jump{u_h}}
\end{align*}
we control the third term in the following way
\begin{align*}
&\frac{\TS}{2 |K|} \sum_{K} \left(\sum_{\sigma \in \facesK} |\sigma| \left( f_{\sigma}^{GRP,New} - f(u_K) \right)\I_{d\times 1} \cdot \vn_K  \right)^2 \br
\aleq \ & \frac{\TS}{ |K|} \sum_{K} \left( \sum_{\sigma \in \facesK}|\sigma|\left( f_{\sigma}^{GRP,New} - f(u_K)  \right)^2 \right) \cdot  \left( \sum_{\sigma \in \facesK}|\sigma| \right) \br
\RevTwo{\aleq} \ & \RevTwo{\frac{\TS}{ h} \sum_{K} \left( \sum_{\sigma \in \facesK}|\sigma|\left( f_{\sigma}^{GRP,New} - f(u_K)  \right)^2 \right)  }\br
\RevTwo{\aleq} \ & \RevTwo{\frac{\TS}{ h} \intfacesintB{ \left( f_{\sigma}^{GRP,New}- f(u_K)  \right)^2 + \left( f_{\sigma}^{GRP,New}- f(u_L)  \right)^2 } }\br
\aleq \ &  \frac{\TS}{h} \intfacesintB{ \left( f_{\sigma}^{GRP,New} - f_{\sigma}^{EC}(u_h) \right)^2 + \left( f_{\sigma}^{EC}(u_h)- f(u_K)  \right)^2 + \left( f_{\sigma}^{EC}(u_h)- f(u_L)  \right)^2 } \br
\aleq \ &  \frac{\TS}{h} \intfacesint{ \left[ \left(  \left(1+\frac{\TS}{h} \right) \jump{u_h} \right)^2  + \jump{u_h}^2  \right]}
\aleq \frac{\TS}{ h} \left(1+\frac{\TS}{h} \right)^2 \intfacesint{  \jump{u_h}^2 }.
\end{align*}
\RevTwo{Thus, we obtain from \eqref{EB-1} and  \eqref{assumption-h-dt} the following estimate
\begin{align*}
\intTd{ D_t \eta_h}&\leq - C_1\intfacesint{  \abs{ \jump{u_h}}^3 } +  C_3\frac{\TS}{h}  \intfacesint{  \Abs{\jump{u_h}_{\sigma}}^2 },
\end{align*}
where $C_3$ is a positive constant only dependent on $\Ov{u}$ and $C_2$.

Further, applying Young's inequality we have
\begin{align*}
C_3 \frac{\TS}{h}  \intfacesint{  \Abs{\jump{u_h}_{\sigma}}^2 }
&\aleq C_3 \frac{\TS}{h} \left( \intfacesint{ 1 }   \right)^{1/3}\left( \intfacesint{  \Abs{\jump{u_h}_{\sigma}}^3 }   \right)^{2/3} \aleq C_3\frac{\TS}{h^{4/3}} \left( \intfacesint{  \Abs{\jump{u_h}_{\sigma}}^3 }   \right)^{2/3} \br
&\RevTwo{\leq C_4\frac{\TS}{h^{4/3}} \left( \frac1{\delta^2}+ \delta\intfacesint{  \Abs{\jump{u_h}_{\sigma}}^3 }   \right),}
\end{align*}
where $C_4$ is a positive constant only dependent on $C_3$, and $\delta$ is a positive constant independent on $\TS$ and $h$.
Setting $\delta = C_1/(2C_2 \cdot C_4)$ we obtain $C_4\frac{\TS}{h^{4/3}} \delta \leq C_1/2$ and
\begin{align*}
\intTd{ D_t \eta_h} &\leq \frac{4C_2^2C_4^2}{C_1^2} \cdot \frac{\TS}{h^{4/3}} - \frac{C_1}2 \intfacesint{  \abs{ \jump{u_h}}^3 } \aleq \frac{\TS}{h^{4/3}} -  \intfacesint{  \abs{ \jump{u_h}}^3 } ,
\end{align*}
}
which gives
\begin{align*}
\intTaufacesint{  \abs{ \jump{u_h}}^3 } &\aleq  \frac{\TS \cdot T}{h^{4/3}} + \intTd{ \eta_h(0,\cdot)} \aleq 1
\end{align*}
and concludes the proof.
\end{proof}

%

As a byproduct we obtain the entropy inequality.
\begin{Corollary}[Entropy inequality]\label{lem-EI}
Let $u_h$ be the numerical solution of the \RevTwo{stabilized} GRP method \eqref{scheme}. Under  assumptions \eqref{assumption} and  \eqref{assumption-h-dt} it holds for all $\varphi_h \in Q_h, \varphi_h \geq 0$
\begin{align}\label{EI}
& \intTd{ \varphi_h  D_t \eta_h}- \intfacesint{ q_{\sigma}^{GRP,New} \jump{\varphi_h} } \aleq \frac{\TS }{h^{4/3}}.
\end{align}
\end{Corollary}

\section{Convergence}\label{Convergence}
Having shown the stability of the numerical approximations obtained by the \RevTwo{stabilized} GRP scheme \eqref{scheme} we proceed to analyze its consistency and then the convergence.

\subsection{Consistency}
Before deriving the consistency formulation we formulate a useful equality.

\begin{Lemma}\label{lem-td}
Let $\tau \in [t_n, t_{n+1})$. Then for any $r_h \in L_{\TS}(0,T), \,  \phi \in W^{1,\infty}(0,T)$  it holds
\begin{align}
&\left[ \phi r_h \right]_{t=0}^{t=\tau} - \int_0^{t_n} \left( \phi D_t r_h + r_h \partial_t \phi\right)\dt  = \int_0^{t_n} \left( D_t \phi(t-\TS) - \partial_t \phi \right) \, r_h \dt \br
&\quad \quad+  \int_{t_n}^{t_{n+1}}r_h(t_n)  \frac{\phi(\tau) -   \phi(t-\TS)}{\TS} \dt - \int_0^{\TS} r_h(t_0)  \frac{\phi(t_0)  - \phi(t-\TS) }{\TS} \dt.
\end{align}
\end{Lemma}

\begin{proof}
With the definition of $D_t$ we have
\begin{align*}
 & \int_0^{t_n} \phi D_t r_h \dt  = \int_0^{t_n} \phi(t) \cdot \frac{r_h(t+\TS) - r_h(t)}{\TS} \dt  \br
= & \frac1{\TS} \int_0^{t_n} \phi(t) \, r_h(t+\TS) \dt  - \frac1{\TS} \int_0^{t_n} \phi(t) \, r_h(t) \dt = \frac1{\TS} \int_{\TS}^{t_{n+1}} \phi(t-\TS) \, r_h(t) \dt  - \frac1{\TS} \int_0^{t_n} \phi(t) \, r_h(t) \dt \br
= & - \int_0^{t_n} D_t \phi(t-\TS) \, r_h(t) \dt -  \frac1{\TS} \int_0^{\TS} \phi(t-\TS) \, r_h(t) \dt  + \frac1{\TS} \int_{t_n}^{t_{n+1}} \phi(t-\TS) \, r_h(t) \dt\br
= & - \int_0^{t_n} D_t \phi(t-\TS) \, r_h(t) \dt -  \frac{r_h(t_0)}{\TS} \int_0^{\TS} \phi(t-\TS)\dt  + \frac{r_h(t_n) }{\TS} \int_{t_n}^{t_{n+1}} \phi(t-\TS) \dt.
\end{align*}
Combining it with
\begin{align*}
(\phi r_h)(\tau) - \frac{r_h(t_n) }{\TS} \int_{t_n}^{t_{n+1}} \phi(t-\TS) \dt & = \phi(\tau) r_h(t_n) - \frac{r_h(t_n) }{\TS} \int_{t_n}^{t_{n+1}} \phi(t-\TS) \dt \br
& =  \int_{t_n}^{t_{n+1}} r_h(t_n)  \frac{\phi(\tau) -   \phi(t-\TS)}{\TS} \dt \br
(\phi r_h)(t_0) - \frac{r_h(t_0)}{\TS} \int_0^{\TS} \phi(t-\TS)  \dt  & = \phi(t_0) r_h(t_0)  - \frac{r_h(t_0)}{\TS} \int_0^{\TS} \phi(t-\TS)  \dt  \br
& =   \int_0^{\TS} r_h(t_0) \frac{\phi(t_0)  - \phi(t-\TS) }{\TS} \dt
\end{align*}
we finish the proof.
\end{proof}

\begin{Lemma}[Consistency I]\label{consistency-1}
Let $u_h$ be the numerical solution of the \RevTwo{stabilized} GRP method \eqref{scheme}. Under assumptions \eqref{assumption} and \eqref{assumption-h-dt}
 it holds for all $\tau \in (0,T)$ and all $\phi \in W^{2,\infty}((0,T)\times \tor)$
\begin{equation}\label{cf-eq}
\left[ \intTd{ u_h \phi} \right]_{t=0}^{t=\tau}= \intTauTdB{ u_h \partial_t \phi +   f_h \cdot \nabla_x \phi } +  e_{u}(\phi,h,\tau), \quad f_h = f(u_h).
\end{equation}
The consistency error $e_u$ satisfies
\begin{align}
\abs{e_{u}(\phi,h,\tau)} \leq C_{u} h^{1/3}.
\end{align}
Here the constant $C_u$ depends solely on $T, \norm{\phi}_{W^{2,\infty}((0,T)\times \tor)}$.		
\end{Lemma}

\begin{proof} Analogously to \cite[Section 10.3]{Feireisl-Lukacova-Mizerova-She:2021} we prove \eqref{cf-eq} in two steps. Let $\tau \in [t_n, t_{n+1})$.

	\paragraph{Step 1 - time derivative term:} Applying Lemma \ref{lem-td} we obtain
\begin{align}
&\left[ \intTd{\phi u_h} \right]_{t=0}^{t=\tau} - \int_0^{t_n} \intTd{ \phi D_t u_h  }\dt - \intTauTd{ u_h \partial_t \phi }  \br
=& \int_0^{t_n} \intTd{\left( D_t \phi(t-\TS) - \partial_t \phi \right) \, u_h} \dt +  \int_{t_n}^{t_{n+1}} \intTd{u_h(t_n) \frac{\phi(\tau) -  \phi(t-\TS)}{\TS} }\dt\br
&\quad  - \int_0^{\TS} \intTd{u_h(t_0) \frac{\phi(t_0)  - \phi(t-\TS) }{\TS} }\dt -  \int_{t_n}^{\tau}\intTd{ u_h \partial_t \phi }  \dt\br
=:& \sum_{i=1}^4 E_t^{i}(u_h).
\end{align}
Further, we estimate \RevTwo{$E_t^{i}(u_h), i=1,2,3,4,$} with
\begin{align*}
& \abs{E_t^{1}(u_h)} \aleq \TS \norm{\phi}_{W^{2,2}(0,T;L^1(\tor))}, \ \abs{E_t^{2}(u_h)} +  \abs{E_t^{3}(u_h)} +\abs{E_t^{4}(u_h)} \aleq \TS \norm{\phi}_{W^{1,2}(0,T;L^1(\tor))}.
\end{align*}

	\paragraph{Step 2 - convective terms:}
	According to \cite[Theorem 3.1]{Lukacova-Yuan} we have
	\begin{subequations}
	\begin{align}
	& \intTauTd{ r_h \cdot \nabla_x \phi } - \int_0^{t_n}\intfacesint{ r_{\sigma}\jump{\Piq \phi}} \dt= \sum_{i=1}^3 E_{i}(r_h),\quad r_h \in Q_h,\\
	&  E_x^{1}(r_h) = -\int_0^{\tau} \intfacesint{ \jump{r_h} (\phi-\avs{\Piq \phi})} \dt,\\
	& E_x^{2}(r_h) =  \int_0^{\tau}\intfacesint{\Big(\avs{r_h}-r_{\sigma}\Big) \jump{\Piq \phi}}\dt,\\
	& E_x^{3}(r_h) =   \int_{t_n}^{\tau}\intTd{ r_h \cdot \nabla_x \phi }\dt.
	\end{align}
	\end{subequations}
	 Further, applying the interpolation error estimates and \eqref{assumption} we get the following estimate \RevTwo{of $\abs{E_x^i(r_h)}, i=1,2,3,$ with $r_h = f_h$ and $r_{\sigma} = f_{\sigma}^{GRP,New}$}:
\begin{align*}
\abs{E_x^1(f_h)} &  \aleq h \norm{\nabla_x \phi}_{L^{\infty}((0,\tau)\times\tor)} \int_0^{\tau} \intfacesint{ \abs{\jump{u_h}} \abs{\avs{u_h}} } \dt \aleq h  \int_0^{\tau} \intfacesint{ \abs{\jump{u_h}} } \dt  \br
& \aleq  h\left(  \int_0^{\tau} \intfacesint{\abs{\jump{u_h}}^3} \dt \right)^{1/3} \left(  \int_0^{\tau} \intfacesint{1^{3/2}} \dt \right)^{2/3} \br
& \aleq  h\left(  \int_0^{\tau} \intfacesint{\abs{\jump{u_h}}^3} \dt \right)^{1/3} h^{-2/3}\aleq h^{1/3}, \br
\abs{E_x^{2}(f_h)} & \aleq h  \norm{\nabla_x \phi}_{L^{\infty}((0,\tau)\times\tor)} \left(1+\frac{\TS}{h} \right)   \int_0^{\tau}\intfacesint{\abs{\jump{u_h}}}\dt \aleq h^{1/3}, \br
\abs{E_x^{3}(f_h)} & \aleq \TS \norm{\nabla_x \phi}_{L^{\infty}((0,\tau)\times\tor)}.
\end{align*}
\RevTwo{Note that we have used $ \int_0^{\tau} \intfacesint{\abs{\jump{u_h}}^3} \dt \aleq 1$, cf. Lemma \ref{lem-UB} in the last inequality of the estimates of $\abs{E_x^1(f_h)}$. }

Finally, \RevTwo{applying the fact that $u_h$ is the numerical solution of the stabilized GRP method \eqref{scheme}, i.e.
\begin{align*}
\int_0^{t_n} \intTd{ \phi D_t u_h  }\dt = \int_0^{t_n} \intTd{ \Piq\phi D_t u_h  }\dt = \int_0^{t_n}\intfacesint{ f_{\sigma}^{GRP,New}\jump{\Piq \phi}} \dt
\end{align*}}
we complete the proof with $e_u = \sum_{i=1}^4 E_t^{i}(u_h)-\sum_{i=1}^3 \RevTwo{E_x^{i}(f_h)}$.	
\end{proof}

The consistency formulation of the entropy inequality \eqref{EI} can be proved in an analogous way as Lemma \ref{consistency-1}. We leave the details for interested readers.
\begin{Lemma}[Consistency II]\label{consistency-2}
Let $u_h$ be the numerical solution of the \RevTwo{stabilized} GRP method \eqref{scheme}. Let assumption \eqref{assumption} and
\begin{align}\label{assumption-h-dt-1}
(\TS,h) \in (0,1]^2, \quad \TS \leq C_2 \, h^{p},  \quad p > 4/3,  \quad  C_2 \mbox{ is a constant independent of } \TS \mbox{ and } h.
\end{align}
hold.

Then it holds for all $\tau \in (0,T)$ and all $\varphi \in W^{2,\infty}((0,T)\times \tor), \varphi \geq 0$
\begin{equation}\label{cf-entropy}
	\left[ \intTd{ \eta_h \varphi }  \right]_{t=0}^{t=\tau} \leq \intTauTdB{ \eta_h \partial_t \varphi + q_h \cdot \nabla_x \varphi} + e_{\eta}(\varphi,h,\tau),\quad q_h = q(u_h).
	\end{equation}
with the consistency error $e_\eta$ satisfying
\begin{align}
\abs{e_{\eta}(\varphi,h,\tau)} \leq C_{\eta} (h^{1/3}+\TS h^{-4/3}).
\end{align}
Here the constant $C_{\eta}$ depends solely on $T, \norm{\phi}_{W^{2,\infty}((0,T)\times \tor)}$.		
\end{Lemma}

\subsection{Convergence}\label{sec-convergence}
We are now ready to prove the main results.
To be precise we recall the definition of the weak and measure-valued solution of the Burgers equation \eqref{PDE}. 

\begin{Definition}
\rm ({\bf Weak solution}) 
\label{def:DW}
We say that $u\in L^{\infty}((0,T)\times\tor)$  is a weak solution of the Burgers equation \eqref{PDE} if the following hold for any  $\tau \in (0,T)$:
%
\begin{align}\label{eq:DWsolution}
\left[\intTd{ u(t,\cdot)\phi} \right]_{t = 0}^{t=\tau}
& = \intTauTdB{ u \partial_t  \phi +   f(u) \cdot \Grad  \phi}
\end{align}
for all $\phi \in W^{1,\infty}((0,T)\times \tor)$ and

	\begin{align}
	\left[\intTd{ \eta(u(t,\cdot))\varphi} \right]_{t = 0}^{t=\tau}
& \leq \intTauTdB{ \eta(u) \partial_t  \varphi +   q(u) \cdot \Grad  \varphi}
	\end{align}
	for all $\varphi \in W^{1,\infty}((0,T)\times \tor), \varphi \geq 0$.
\end{Definition}

In order to show the convergence of newly developed GRP scheme we need the following result due to DiPerna on characterisation of the measure-valued solution of the Burgers equation.

\begin{Definition}\label{def_MVS}
\rm ({\bf Measure-valued solution}) 
A space-time parametrized probability measure, the Young measure, \RevOne{$\mathcal{V}_{t,x}: (t,x)\in (0,T)\times \R \to \mathcal{P}(\R)$} is called an admissible
measure valued solution of \eqref{PDE} if the following holds for any $\tau \in (0,T)$
$$
\left[ \int_{\tor} \langle \mathcal{V}_{t,x}; \tilde{u}(t, \cdot) \rangle \phi (t) \dx \right]_{t=0}^{t=\tau}
= \int_{0}^\tau  \int_{\tor}  \langle \mathcal{V}_{t,x}; \tilde{u} \rangle \partial_t \phi (t) \mbox{d} x \,\mbox{d} t
+ \int_{0}^\tau \int_{\tor} \langle \mathcal{V}_{t,x}; f(\tilde{u}) \rangle \cdot \nabla_x \phi  \, \mbox{d} x \, \mbox{d} t
$$
for all $\phi \in W^{1,\infty}((0,T)\times \tor)$ and
$$
\left[\intTd{\langle \mathcal{V}_{t,x}; {\eta}(\tilde{u}(t,\cdot)) \rangle \varphi} \right]_{t = 0}^{t=\tau}
\leq \intTauTdB{\langle \mathcal{V}_{t,x}; {\eta}(\tilde{u}) \rangle \partial_t  \varphi +  \langle \mathcal{V}_{t,x}; {q} (\tilde{u})   \rangle \cdot \Grad  \varphi}
$$
for all $\varphi \in W^{1,\infty}((0,T)\times \tor), \varphi \geq 0$.
\RevOne{We note that $\tilde{u}$ denotes the independent variable on which the measure $\mathcal{V}_{t,x}$ acts. }
\end{Definition}

%
%

We proceed by recalling the following result of DiPerna~\cite{DiPerna:1985a} on the characterisation of the Young measure.

\begin{Theorem}\cite[Theorem~4.1 and 4.2]{DiPerna:1985a}
\label{MVS}
Let $u_0 \in L^{\infty}(\tor)$ and let $\mathcal{V}_{t,x}$ be admissible measure-valued solution in the sense of Definition~\ref{def_MVS}.
Assume that
\begin{itemize}
\item[i)] there exists a constant $C>0$ such that for almost all $t \in [0,T]$
$$
\int_{\tor} \langle \mathcal{V}_{t,x}; | \tilde{u}| \rangle \mbox{d} x \leq C;
$$
\item[ii)]
$$
\lim_{t\to 0} \frac 1 t  \int_0^t \int_{\tor} \langle  \mathcal{V}_{s,x}; \tilde{u} \rangle  \phi(x) \mbox{d} x \mbox{d}s
= \int_{\tor}  {\tilde{u}}_0  \phi(x)\mbox{d} x
$$
for all $\phi \in C^1(\tor);$
\item[iii)]
$$
\lim_{t\to 0} \frac 1 t  \int_0^t \int_{\tor} \langle  \mathcal{V}_{s,x}; \tilde{\eta} \rangle \mbox{d} x \mbox{d}s
\leq \int_{\tor} \eta(u_0(x)) \mbox{d} x
$$
holds for one strictly convex function $\eta: \Bbb R \to \Bbb R$ with $\eta(0) = 0.$
\end{itemize}
Then it holds
$$
\mathcal{V}_{t,x} = \delta_{u(t,x)} \qquad \mbox{ for a.a } (t,x) \in [0,T]\times\tor.
$$
\end{Theorem}

This establishes
sufficient conditions which guarantee that an approximating sequence
contains a subsequence which converges to the unique Kru\v{z}kov weak, entropy solution.

\begin{Theorem}[Weak convergence of the GRP scheme]\label{thm-weak}
Let $\{u_h\}_{h\searrow 0}$ be the family of numerical solutions obtained by the \RevTwo{stabilized} GRP method \eqref{scheme}. Let assumptions \eqref{assumption} and \eqref{assumption-h-dt-1} hold.

Then, up to a subsequence, the GRP solution $\{u_{h}\}$ converges for $h\to 0$ in the following sense
\begin{equation}
u_{h}  \longrightarrow \  u \quad \mbox{weakly-}(*) \ \mbox{in} \ L^{\infty}((0,T)\times\tor)),
\end{equation}
where \RevTwo{$u$ is the unique}  weak entropy solution of the Burgers equation, cf.~Definition \ref{def:DW}.
\end{Theorem}

\begin{proof}
Applying the fundamental theorem on Young measure \cite{Ball:1989} we obtain from \eqref{assumption} that, up to a subsequence, there exist a parametrized probability measure $\{\mathcal{V}_{t,x}\}_{(t,x)\in(0,T)\times \tor}$ such that
\begin{align}
& u_h  \longrightarrow \ \langle \mathcal{V}_{t, x};\, \tilde{u} \rangle  \quad \mbox{weakly-}(*) \ \mbox{in} \ L^{\infty}((0,T)\times\tor), \\
& (f,\eta,q)(u_h)  \longrightarrow \ \langle \mathcal{V}_{t, x};\, {(f,\eta,q)(\tilde{u})} \rangle \quad \mbox{weakly-}(*) \ \mbox{in} \ L^{\infty}((0,T)\times\tor;\R^3).
\end{align}
Further, thanks to the consistency formula, cf.~Lemmas \ref{consistency-1} and \ref{consistency-2}, we know that the Young measure $\{\mathcal{V}_{t,x}\}_{(t,x)\in(0,T)\times \tor}$ satisfies
\begin{align}\label{DM-1}
\left[\intTd{ \langle \mathcal{V}_{t, x};\, \tilde{u} \rangle \, \phi} \right]_{t = 0}^{t=\tau}
& = \intTauTdB{ \langle \mathcal{V}_{t, x};\, \tilde{u} \rangle \, \partial_t  \phi +   \langle \mathcal{V}_{t, x};\, f(\tilde{u}) \rangle \cdot \Grad  \phi} \\ \label{DM-2}
 \left[\intTd{ \langle \mathcal{V}_{t, x};\, \eta(\tilde{u}) \rangle \, \varphi} \right]_{t = 0}^{t=\tau}
& \leq \intTauTdB{ \langle \mathcal{V}_{t, x};\, \eta(\tilde{u}) \rangle \, \partial_t  \varphi +   \langle \mathcal{V}_{t, x};\, q(\tilde{u}) \rangle \cdot \Grad  \varphi}
\end{align}
for all $\phi \in W^{1,\infty}((0,T)\times \tor)$, $\varphi \in W^{1,\infty}((0,T)\times \tor), \varphi \geq 0$. Consequently, we have a weak convergence of our numerical solution $\{ u_h \}_{h \searrow 0}$ to a measure-valued solution $\mathcal{V}_{t,x}.$ Now, in order to obtain the convergence to a unique weak entropy solution we apply Theorem~\ref{MVS}. Indeed,  following  \cite[Theorems 3.8 and 3.10]{Kroner-Noelle-Rokyta:1995}  we can analogously verify that all properties of Theorem~\ref{MVS} hold also for the GRP scheme. Consequently, we obtain that $\mathcal{V}_{t,x} = \delta_{u(t,x)}$ and we have shown that up to a subsequence our numerical solutions converge  to a weak entropy solution. This concludes the proof.
\end{proof}

Further, applying the weak-strong uniqueness principle, analogously as \cite[Section 7]{Feireisl-Lukacova-Mizerova-She:2021} we obtain the following strong convergence results.
\begin{Theorem}[Strong convergence]\label{thm-strong}
Let $\{u_h\}_{h\searrow 0}$ be the family of numerical solutions obtained by the \RevTwo{stabilized} GRP method \eqref{scheme}. Let assumptions \eqref{assumption} and \eqref{assumption-h-dt-1} hold.

Then the following holds:
\begin{itemize}
\item {\bf Classical solution}  \quad  Let $u\in C^1([0,T]\times \tor)$. Then $u$ is a classical solution to the Burgers equation  \eqref{PDE} with initial data $u_0\in C^1(\tor)$, and
\begin{equation*}
u_{h}  \longrightarrow \  u \quad \mbox{in} \ L^{p}((0,T)\times \tor)
\end{equation*}
for any $p\in[1,\infty)$.

\item {\bf Strong solution}  \quad Supposed that the Burgers equation \eqref{PDE} admits a strong solution $u$ in the class $W^{1,\infty}((0,T)\times \tor)$ emanating from initial data $u_0\in W^{1,\infty}(\tor)$. Then it holds
\begin{equation*}
u_{h}  \longrightarrow \  u \quad \mbox{in} \ L^{p}((0,T)\times \tor) 
\end{equation*}
for any $p\in[1,\infty)$.
\end{itemize}
\end{Theorem}

\

\section*{Acknowledgements}


\noindent This manuscript is dedicated to the 65th birthday of Gerald Warnecke, our colleague and close friend. We want to express our deep thank
for his continuous support over the years and
his time he invested in fruitful joint research discussions.

In addition we would like also to dedicate this paper to the memory of Jiequan Li, the founder of the GRP method. We will never forget his friendship and enthusiasm when discussing with us mathematics.

\section*{Compliance with Ethical Standards}

\paragraph{Conflict of Interest}
On behalf of all authors, the corresponding author states that there is no conflict of interest.

\paragraph{Ethics approval and informed consent}
On behalf of all authors, the corresponding author confirms that the appropriate ethics review and informed consent have been followed.

\section*{Declarations}
\paragraph{Funding}  The work of M.L. and Y.Y. was partially funded by the Gutenberg Research College
 and by Chinesisch-Deutschen Zentrum f\"ur Wissenschaftsf\"orderung\chinese{(中德科学中心)} - Sino-German project number GZ1465.
M.L. is grateful to the Mainz Institute of Multiscale Modelling and 
SPP 2410 Hyperbolic Balance Laws in Fluid Mechanics: Complexity, Scales, Randomness (CoScaRa) for supporting her research.

\paragraph{Data Availability} There is no data associated with this manuscript.


\appendix

\section{Entropy-stable factor of RP flux}\label{sec-ES-RP}
The goal of this section is to estimate the entropy-stable factor $D_{\sigma}^{RP}$, defined in \eqref{DRP}. Let $r_h = R_Vu_h$. We start by deriving the expression of $D_{\sigma}^{RP}$ with Table \ref{t-GRP}.

Let us consider a face $\sigma = K|L$. Then for $r_{K} = r_{L}$ we simply set $D_{\sigma}^{RP} = \avs{r_h} + \jump{r_h}$.  Otherwise, some calculations directly give
\begin{equation}\label{DRP1}
D_{\sigma}^{RP}(r_h)  = \begin{cases}
\left(6 \avs{r_h} - \jump{r_h} \right)/12  & \mbox{for Cases 1 and  3}, \\
-\left(6 \avs{r_h} + \jump{r_h} \right)/12 & \mbox{for Cases 2 and  5}, \\
\left( 12 \avs{r_h}^2 + \jump{r_h}^2\right) / (24\jump{r_h})  \  & \mbox{for Case 4}.
\end{cases}
\end{equation}


\begin{Lemma} \label{lem-DRP}
	It holds that  	
	\begin{equation}\label{es-DRP}
	D_{\sigma}^{RP}(r_h) \approx \abs{\avs{r_h}} +  \abs{\jump{r_h}}.
	\end{equation}
\end{Lemma}

\begin{proof}
	Thanks to \eqref{DRP1} in what follows we prove \eqref{es-DRP} case by case.
	\begin{itemize}
		\item Consider Case 1:  $\jump{r_h} < 0, \avs{r_h} > 0$. Then it holds
		\begin{equation*}
		D_{\sigma}^{RP} = \frac12 \abs{\avs{r_h}} + \frac1{12} \abs{\jump{r_h}} > 0.
		\end{equation*}
		
		Analogously to analyze Case 2: $\jump{r_h} < 0, \avs{r_h} \leq 0$, we obtain $D_{\sigma}^{RP} = \frac12 \abs{\avs{r_h}} + \frac1{12} \abs{\jump{r_h}}$ .
		
		\item Consider Case 3:  $0 < r_{\sigma}^- < r_{\sigma}^+$. then it holds $2\avs{r_h} > \jump{r_h} \geq 0$ yielding
		\begin{equation*}
		\frac16 \abs{\avs{r_h}} + \frac1{12}  \abs{\jump{r_h}} < \frac13 \avs{r_h} < D_{\sigma}^{RP} \leq \frac12 \avs{r_h} \leq \frac12 \abs{\avs{r_h}} +  \abs{\jump{r_h}}.
		\end{equation*}
		
		Analogously to analyze Case 5: $r_{\sigma}^- < r_{\sigma}^+ < 0$, we obtain $D_{\sigma}^{RP} \approx \abs{\avs{r_h}} +  \abs{\jump{r_h}}$ .
		
		\item Consider Case 4: $r_{\sigma}^- < 0 < r_{\sigma}^+$. We have $\jump{r_h} > 2 \abs{\avs{r_h}} \geq 0$, resulting to
		\begin{align*}
		& D_{\sigma}^{RP}  \geq \frac1{24}\abs{\jump{r_h}} > \frac{1}{48} \left( \abs{\jump{r_h}} + \abs{\avs{r_h}} \right), \\
		& D_{\sigma}^{RP} < \frac16  \abs{\jump{r_h}}  \leq \frac16  \abs{\jump{r_h}} + \abs{\avs{r_h}}.
		\end{align*}
	\end{itemize}
	Altogether, we complete the proof.
\end{proof}

\end{document}